\documentclass[11pt,letterpaper]{article}

\usepackage{amsfonts}
\usepackage{amsthm}
\usepackage{amsmath,amssymb}
\usepackage{mathtools}

\usepackage{graphicx}
\usepackage{subfig}
\usepackage{caption,subcaption}
\usepackage{xcolor}

\usepackage{tikz}
\usetikzlibrary{arrows.meta,calc,positioning,backgrounds,fit,decorations.pathreplacing}

\tikzset{
  redpt/.style   = {circle, fill=red!75, draw=black, inner sep=1.7pt},
  bluept/.style  = {circle, fill=blue!70, draw=black, inner sep=1.7pt},
  match/.style   = {thick},
  shiftedge/.style = {thick, dashed},
  dircycle/.style  = {->, >=Latex, thick},
  levcurve/.style  = {thick},
  regionA/.style = {fill=green!20, draw=green!50!black, opacity=.45},
  regionB/.style = {fill=orange!25, draw=orange!70!black, opacity=.45},
  regionC/.style = {fill=cyan!20, draw=cyan!60!black, opacity=.45},
  note/.style    = {font=\small},
  lab/.style     = {font=\small},
  axisarrow/.style = {->, >=Latex, thick}
}

\allowdisplaybreaks

\usepackage{url}
\usepackage[colorlinks=true,citecolor=black,linkcolor=black,urlcolor=blue]{hyperref}

\usepackage{fullpage}

\usepackage[mathlines]{lineno}

\newcommand*\patchAmsMathEnvironmentForLineno[1]{%
  \expandafter\let\csname old#1\expandafter\endcsname\csname #1\endcsname
  \expandafter\let\csname oldend#1\expandafter\endcsname\csname end#1\endcsname
  \renewenvironment{#1}%
     {\linenomath\csname old#1\endcsname}%
     {\csname oldend#1\endcsname\endlinenomath}}%
\newcommand*\patchBothAmsMathEnvironmentsForLineno[1]{%
  \patchAmsMathEnvironmentForLineno{#1}%
  \patchAmsMathEnvironmentForLineno{#1*}}%
\AtBeginDocument{%
\patchBothAmsMathEnvironmentsForLineno{equation}%
\patchBothAmsMathEnvironmentsForLineno{align}%
\patchBothAmsMathEnvironmentsForLineno{flalign}%
\patchBothAmsMathEnvironmentsForLineno{alignat}%
\patchBothAmsMathEnvironmentsForLineno{gather}%
\patchBothAmsMathEnvironmentsForLineno{multline}%
}

\allowdisplaybreaks

\newtheorem{theorem}{Theorem}[section]

\newtheorem{lemma}[theorem]{Lemma}
\newtheorem{corollary}[theorem]{Corollary}

\newcommand{\M}{\ensuremath{\mathcal{M}}}

\title{Characterization of bichromatic maximum-sum matchings of points and matching equilibrium}

\author{
	Oscar Chacón-Rivera\thanks{Pontificia Universidad Católica de Chile, Facultad de Matemáticas, Chile. {\tt opchacon@mat.uc.cl}.}
}

\begin{document}

\maketitle

\begin{abstract}
We study maximum-sum red-blue matchings and matching equilibrium for finite planar point sets. For a red-blue perfect matching $M = \{(a_i,b_i) : 1 \le i \le n\}$, we define the gain of a directed red cycle as the change in total weight produced by cyclically shifting the corresponding blue partners. We prove that $M$ is maximum-sum if and only if every directed red cycle has nonpositive gain, and we derive a geometric sufficient condition for optimality from cyclic intersections of distance-difference regions. We then characterize balanced matchings, in which all red-blue perfect matchings have the same total weight. Equilibrium is shown to be equivalent to vanishing cycle gains, to an additive form of the distance matrix, and to a common level-set condition for distance-difference functions. In the squared Euclidean case this yields an orthogonality classification, while in the Euclidean case it yields a hyperbolic level-set description and a collinear-separation classification in the nondegenerate setting.
\end{abstract}

\section{Introduction}
Let $R$ and $B$ be two point sets in the Euclidean plane with $|R| = |B|$. The points in $R$ are called \textit{red} points, and those in $B$ are called \textit{blue} points. A \textit{matching} $\M$ of $R \cup B$ is a partition of $R \cup B$ into $n$ pairs such that each pair consists of a red point and a blue point. A point $p \in R$ and a point $q \in B$ are matched by $M$ if and only if the (unordered) pair $(p,q)$ is in the matching $M$.

Given a metric or a semi-metric function $d : \mathbb{R}^2 \times \mathbb{R}^2 \rightarrow \mathbb{R}_{\ge 0}$, we say that a matching $\M$ is {\em max-sum} if it maximizes $\sum_{(p,q)\in\M} d(p,q)$ among all matchings of $R$ and $B$. Recall that a function is a semi-metric if it satisfies all the properties of a metric function except for the triangle inequality. 

Chacón-Rivera and Pérez-Lantero~\cite{Chacon2026} characterized a maximum-sum matching in terms of $H$-sets and $h$-sets defined as
\begin{align*}
 H(p,q) = & \{ x \in \mathbb{R}^2 : d(p,q') - d(q,q') \le d(p,x) - d(q,x) \},
\end{align*} and 
\begin{align*}
 h(p,q) = & \{ x \in \mathbb{R}^2 : d(p,x) - d(q,x) \le d(p,p') - d(q,p') \},
\end{align*} where $p,q$ are red points, and $p',q'$ are blue points, and $\{ (p,p'),(q,q') \}$ is a maximum-sum matching of those four points, that is, $M$ is a 2-local max-sum matching as defined by Biniaz et al. \cite{Biniaz2025}. This characterization proved useful in simplifying the proof of the common intersection property of disks established by Huemer et al. \cite{Huemer}.

In this article, we consider max-sum matchings between planar colored point sets $R$ and $B$ with $|R| = |B| = n$. In what follows, we generalize the characterization proposed by Chacón-Rivera and Pérez-Lantero~\cite{Chacon2026} to matchings of $n$ red points and $n$ blue points, and then use this characterization to study equilibrium phenomena in red-blue matchings.

\subsection{Related work and motivation}

Geometric matching problems often combine an optimization condition with intersection properties of the objects induced by the selected edges. In the setting of planar point sets, a matched pair naturally induces the disk having the segment joining the two matched points as diameter. Huemer et al.~\cite{Huemer} proved that if $R$ and $B$ are finite point sets in the plane with $|R| = |B|$, and $M$ is a red-blue perfect matching maximizing the total squared Euclidean length of its edges, then all diametral disks induced by the edges of $M$ have a nonempty common intersection. This result showed that a global optimality condition on a matching may force a strong geometric piercing property.

The analogous question for the ordinary Euclidean distance is more subtle. Bereg et al.~\cite{Bereg2} showed that, in the bichromatic Euclidean setting, the disks induced by a maximum-sum matching need not have a common point, although they satisfy weaker intersection properties. In contrast, they proved that for a set of $2n$ uncolored points in the plane, every Euclidean maximum-sum matching induces diametral disks with a nonempty common intersection. More recently, Pérez-Lantero and Seara~\cite{PerezLantero} obtained a bichromatic analogue in terms of ellipses: if $M = \{(r_i,b_i) : 1 \le i \le n\}$ is a Euclidean maximum-sum red-blue matching, then there exists a point $o$ such that \[ \|r_i-o\| + \|b_i-o\| \le \sqrt{2} \cdot \|r_i-b_i\| \] for every $i$.

These intersection results fit naturally into the language of Tverberg graphs. Given a finite point set $P$ and a geometric graph $G$ with vertex set $P$, one says that $G$ is a Tverberg graph if the diametral disks, or balls in higher dimension, induced by the edges of $G$ have a common point. Soberón and Tang~\cite{Soberon} studied this point of view for Hamiltonian structures in the plane, proving that every odd planar point set admits a Hamiltonian cycle which is a Tverberg graph, and that every even planar point set admits a Hamiltonian path with the same property. Pirahmad, Polyanskii, and Vasilevskii~\cite{Pirahmad} refined and extended this framework, obtaining Tverberg Hamiltonian cycles in the plane, Tverberg matchings in higher-dimensional Euclidean spaces, and red-blue variants for diametral balls.

A parallel line of work concerns spanning trees. Abu-Affash, Carmi, and Maman~\cite{Abu} proved that if $T$ is a maximum-weight spanning tree of a finite planar point set, then the diametral disks induced by the edges of $T$ have a common point; more precisely, the center of the smallest enclosing circle of the point set belongs to all these disks. Barabanshchikova and Polyanskii~\cite{Barabanshchikova2} later generalized this phenomenon, proving that max-sum trees in Euclidean spaces are Tverberg graphs and strengthening related intersection results for matchings. Thus, maximum-sum matchings and maximum-weight spanning trees belong to a broader family of extremal geometric graphs whose edges induce commonly pierced disks or balls.

Another important motivation comes from a conjecture of Fingerhut, mentioned by Eppstein~\cite{Eppstein} and motivated by network-design problems of Fingerhut, Suri, and Turner~\cite{Fingerhut}. The conjecture asked whether, for every Euclidean maximum-sum matching $\{(a_i,b_i) : 1 \le i \le n\}$ of $2n$ planar points, there exists a point $o$ such that \[ \|a_i-o\| + \|b_i-o\| \le \frac{2}{\sqrt3}\|a_i-b_i\| \] for every $i$. Equivalently, the ellipses with foci $(a_i,b_i)$ and major axis length $(2/\sqrt3)|a_i-b_i|$ should have a common point. Barabanshchikova and Polyanskii~\cite{Barabanshchikova} confirmed this conjecture by proving the corresponding common-intersection theorem for ellipses induced by a maximum-sum matching. This connects maximum-sum matchings not only with diametral disks, but also with families of ellipses and more general convex sets induced by matching edges.

Maximum-length matchings also appear in classical geometric optimization. For example, Fekete and Meijer~\cite{Fekete} studied relations between maximum matchings, minimum stars, and Steiner stars. More recently, local optimality conditions for Euclidean maximum matchings were investigated by Biniaz, Maheshwari, and Smid~\cite{Biniaz2025}. They introduced $k$-local maximum matchings and obtained approximation bounds comparing locally maximum matchings with globally maximum matchings. This local-to-global perspective is complementary to the present article: local optimality asks how much information is needed to approximate global optimality, while our cycle-gain characterization gives an exact condition for global optimality.

The present article follows a structural route. Rather than proving a new piercing theorem for disks or ellipses, we characterize maximum-sum red-blue matchings through cycle gains and through distance-difference regions associated with ordered pairs of red points. This extends the $n = 3$ characterization by Chacón-Rivera and Pérez-Lantero~\cite{Chacon2026} to arbitrary $n$ in a form that separates algebraic optimality from geometric certificates. We then study balanced, or equilibrium, matchings, for which every red-blue perfect matching has the same total weight. The equilibrium condition leads to additive distance matrices and to geometric level-set characterizations. In the squared Euclidean case, this becomes an orthogonality condition between the spans of the two color classes; in the Euclidean case, it yields hyperbolic level-set constraints and a collinear-separation classification in the nondegenerate case.


\section{Bichromatic max-sum matchings}
\label{sec:bichromatic}

\subsection{Characterization of bichromatic max-sum matchings of $n$ red points and $n$ blue points}
\label{subsec:bic-1}

Consider an arbitrary continuous (semi)metric $d : \mathbb{R}^2 \times \mathbb{R}^2 \to \mathbb{R}$ satisfying the following properties:
\begin{enumerate}
 \item[P1] $d(x,y) \ge 0$ for all $x,y \in \mathbb{R}^2$.
 \item[P2] $d(x,x) = 0$ for all $x \in \mathbb{R}^2$.
 \item[P3] $d(x,y) = d(y,x)$ for all $x,y \in \mathbb{R}^2$.
 \item[P4] For any fixed points $p,q \in \mathbb{R}^2$ with $p \neq q$ and any real constant $t$, the level set $\{ z \in \mathbb{R}^2 : d(p,z) - d(q,z) = t \}$ is path-connected.
\end{enumerate} Note that the first three properties are the standard properties of any semimetric, whereas the fourth property is satisfied by semimetrics such as the Euclidean distance (a metric) and the squared Euclidean distance.

Let $R = \{ a_1, a_2, \dots, a_n \}$ be a set of $n$ red points, $B = \{ b_1, b_2, \dots, b_n \}$ a set of $n$ blue points, and suppose that $M = \{ (a_i,b_i) : i = 1, \dots, n \}$ is a matching of $R \cup B$. Define the sets \[ H(a_i,a_j) = \{ x \in \mathbb{R}^2 : d(a_i,b_j) - d(a_j,b_j) \le d(a_i,x) - d(a_j,x) \}, \] and write $h(a_i,a_j) = H(a_j,a_i)$. If $M$ is a \emph{2-local} max-sum, then $H(a_i,a_j) \cap H(a_j,a_i) \neq \emptyset$ for every $i \neq j$. Geometrically speaking, $H(a_i,a_j)$ is the region of every possible point $x$ in the plane for which $\{ (a_i,x), (a_j,b_j) \}$ is a max-sum matching, while $h(a_i,a_j)$ is the region of every possible point $x$ in the plane for which $\{ (a_i,b_i), (a_j,x) \}$ is a max-sum matching. See Figure \ref{fig:Hhsets}. Observe that for the Euclidean distance, the boundaries are distance-difference level curves with foci \(a_i\) and \(a_j\) (see Figure \ref{fig:Hhsets}(a)); for the squared Euclidean distance, the boundaries are parallel lines, and \(S(a_i,a_j)\) is a strip (see Figure \ref{fig:Hhsets}(b)). 

\begin{figure}[t]
\centering
\begin{tikzpicture}[scale=1.0]

\begin{scope}[xshift=-4.2cm]
  \node[note] at (0,-3.2) {(a) Euclidean distance};

  \fill[green!20, opacity=.55]
    plot[domain=-1.9:1.9, samples=120, smooth]
      ({-0.80 - 0.23*\x*\x}, {\x})
    --
    plot[domain=1.9:-1.9, samples=120, smooth]
      ({0.80 + 0.23*\x*\x}, {\x})
    -- cycle;

  \draw[levcurve]
    plot[domain=-1.9:1.9, samples=120, smooth]
      ({-0.80 - 0.23*\x*\x}, {\x});
  \draw[levcurve]
    plot[domain=-1.9:1.9, samples=120, smooth]
      ({0.80 + 0.23*\x*\x}, {\x});

  \node[redpt, label=below:$a_i$] (ai1) at (-2.1,0) {};
  \node[redpt, label=below:$a_j$] (aj1) at ( 2.1,0) {};

  \node[lab] at (0,-2.5) {$S(a_i,a_j)=H(a_i,a_j)\cap h(a_i,a_j)$};
  \node[lab] at (-2.7,2.2) {$h(a_i,a_j)$};
  \node[lab] at ( 2.7,2.2) {$H(a_i,a_j)$};

  \draw[->] (-2.2,1.9) -- (-1.25,1.2);
  \draw[->] ( 2.2,1.9) -- ( 1.25,1.2);
\end{scope}

\begin{scope}[xshift=4.2cm]
  \node[note] at (0,-3.2) {(b) Squared Euclidean distance};

  \fill[green!20, opacity=.55] (-0.9,-2.0) rectangle (0.9,2.0);

  \draw[levcurve] (-0.9,-2.2) -- (-0.9,2.2);
  \draw[levcurve] ( 0.9,-2.2) -- ( 0.9,2.2);

  \node[redpt, label=below:$a_i$] (ai2) at (-2.1,0) {};
  \node[redpt, label=below:$a_j$] (aj2) at ( 2.1,0) {};

  \node[lab] at (0,-2.5) {$S(a_i,a_j)=H(a_i,a_j)\cap h(a_i,a_j)$};
  \node[lab] at (-2.2,1.9) {$h(a_i,a_j)$};
  \node[lab] at ( 2.2,1.9) {$H(a_i,a_j)$};

  \draw[->] (-1.7,1.6) -- (-1.05,1.0);
  \draw[->] ( 1.7,1.6) -- ( 1.05,1.0);
\end{scope}

\end{tikzpicture}
\caption{The regions \(H(a_i,a_j)\) and \(h(a_i,a_j)\), and their intersection
\(S(a_i,a_j)=H(a_i,a_j)\cap h(a_i,a_j)\).}
\label{fig:Hhsets}
\end{figure}
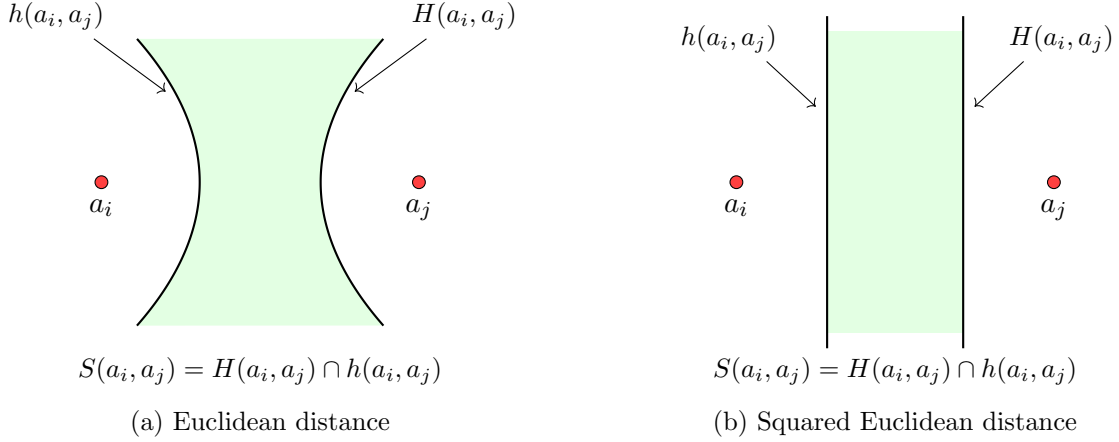

Chacón-Rivera and Pérez-Lantero~\cite{Chacon2026} characterized any max-sum matching of $3$ red points and $3$ blue points in terms of certain intersection of the $H$-sets and the $h$-sets. Precisely,

\begin{theorem}[\cite{Chacon2026}]
\label{thm:3-characterization}
 The matching $\{ (a_i,b_i) : i \in \{1,2,3\} \}$ is a maximum-sum matching of $R \cup B$ if and only if the five intersections \[ H(a_1,a_2) \cap H(a_2,a_3) \cap H(a_3,a_1), \quad h(a_1,a_2) \cap h(a_2,a_3) \cap h(a_3,a_1) \] \[ H(a_1,a_2)\cap h(a_1,a_2), \quad H(a_2,a_3)\cap h(a_2,a_3), \quad \text{and} \quad H(a_3,a_1)\cap h(a_3,a_1) \] are not empty.
\end{theorem}

When $M$ is a max-sum matching, the previous characterization for $n = 3$ implies a local triangular intersection of the $H$-sets in the general case.

\begin{lemma}
\label{lem:3-local-H-intersection}
 Let $n \ge 3$. If $M$ is max-sum, then for every three distinct indices $i,j,k$, \[ H(a_i,a_j) \cap H(a_j,a_k) \cap H(a_k,a_i) \neq \emptyset. \]
\end{lemma}

\begin{proof}
 Fix three distinct indices $i,j,k$, and let us restrict to the sets $R_{ijk} = \{ a_i,a_j,a_k \}$ and $B_{ijk} = \{ b_i,b_j,b_k \}$, and the restricted matching \[ M_{ijk} = \{ (a_i,b_i),(a_j,b_j),(a_k,b_k) \}. \]
 
 We first claim that $M_{ijk}$ is a max-sum matching of $R_{ijk} \cup B_{ijk}$. By contradiction, suppose it is not the case. Then there exists a permutation $\tau$ of $\{i,j,k\}$ such that \[ d(a_i,b_{\tau(i)}) + d(a_j,b_{\tau(j)}) + d(a_k,b_{\tau(k)}) > d(a_i,b_i) + d(a_j,b_j) + d(a_k,b_k). \] Define $\widetilde{M}$ as the matching of $R \cup B$ by replacing \[ (a_i,b_i),(a_j,b_j),(a_k,b_k) \qquad \text{with} \qquad (a_i,b_{\tau(i)}),(a_j,b_{\tau(j)}),(a_k,b_{\tau(k)}) \] and leaving every other pair of $M$ unchanged. Then \[ \sum_{(a,b) \in \widetilde{M}} d(a,b) > \sum_{(a,b) \in M} d(a,b), \] contradicting that $M$ is max-sum. Therefore, $M_{ijk}$ is max-sum on $R_{ijk} \cup B_{ijk}$.

 By Theorem \ref{thm:3-characterization}, applied to the ordered triple $(a_i,a_j,a_k)$, we obtain \[ H(a_i,a_j) \cap H(a_j,a_k) \cap H(a_k,a_i) \neq \emptyset, \] proving the claim. 
\end{proof} 

We now prove an algebraic characterization for the general case $n \ge 4$. For every subset $\{ a_{i_1}, a_{i_2}, \dots, a_{i_m} \}$ of $R$, define $C = (a_{i_1}, a_{i_2} \dots, a_{i_m})$ as the directed (red) cycle of the red vertices $(a_{i_1}, a_{i_2}, \dots, a_{i_m})$, where $i_1, \dots, i_m$ are distinct and $m \ge 2$. Also, define its \emph{cycle gain} \[ \operatorname{gain}(C) = \sum_{j=1}^m d(a_{i_j},b_{i_{j+1}}) - d(a_{i_j},b_{i_j}) \] where $i_{m+1} = i_1$. Thus, $\operatorname{gain}(C)$ measures the change in total weight obtained by replacing the matched pairs $(a_{i_1},b_{i_1}), \dots, (a_{i_m},b_{i_m})$ with the cyclically shifted pairs $(a_{i_1}, b_{i_2}), \dots, (a_{i_m},b_{i_1})$. See Figure \ref{fig:cycle-gain}. 

\begin{figure}[t]
\centering
\begin{tikzpicture}[scale=1.0]

\node[redpt, label=above:$a_1$] (a1) at (0, 2.6) {};
\node[redpt, label=right:$a_2$] (a2) at (2.6, 0) {};
\node[redpt, label=below:$a_3$] (a3) at (0,-2.6) {};
\node[redpt, label=left:$a_4$]  (a4) at (-2.6,0) {};

\node[bluept, label=below:$b_1$] (b1) at (0, 1.0) {};
\node[bluept, label=left:$b_2$] (b2) at (1.0, 0) {};
\node[bluept, label=above:$b_3$] (b3) at (0,-1.0) {};
\node[bluept, label=right:$b_4$]  (b4) at (-1.0,0) {};

\draw[dircycle] (a1) -- (a2);
\draw[dircycle] (a2) -- (a3);
\draw[dircycle] (a3) -- (a4);
\draw[dircycle] (a4) -- (a1);

\draw[match] (a1) -- (b1);
\draw[match] (a2) -- (b2);
\draw[match] (a3) -- (b3);
\draw[match] (a4) -- (b4);

\draw[shiftedge] (a1) -- (b2);
\draw[shiftedge] (a2) -- (b3);
\draw[shiftedge] (a3) -- (b4);
\draw[shiftedge] (a4) -- (b1);

\draw[match] (-5.6,-3.3) -- (-4.6,-3.3);
\node[lab, anchor=west] at (-4.4,-3.3) {original matching edges \((a_i,b_i)\)};
\draw[shiftedge] (2.0,-3.3) -- (3.0,-3.3);
\node[lab, anchor=west] at (3.2,-3.3) {cyclic shift edges \((a_i,b_{i+1})\)};


\end{tikzpicture}
\caption{A directed red cycle \(C=(a_1,a_2,a_3,a_4)\). The gain of \(C\) is the change in total weight
obtained by replacing the solid matching edges \((a_i,b_i)\) with the dashed cyclically shifted edges
\((a_i,b_{i+1})\), where \(b_5=b_1\).}
\label{fig:cycle-gain}
\end{figure}
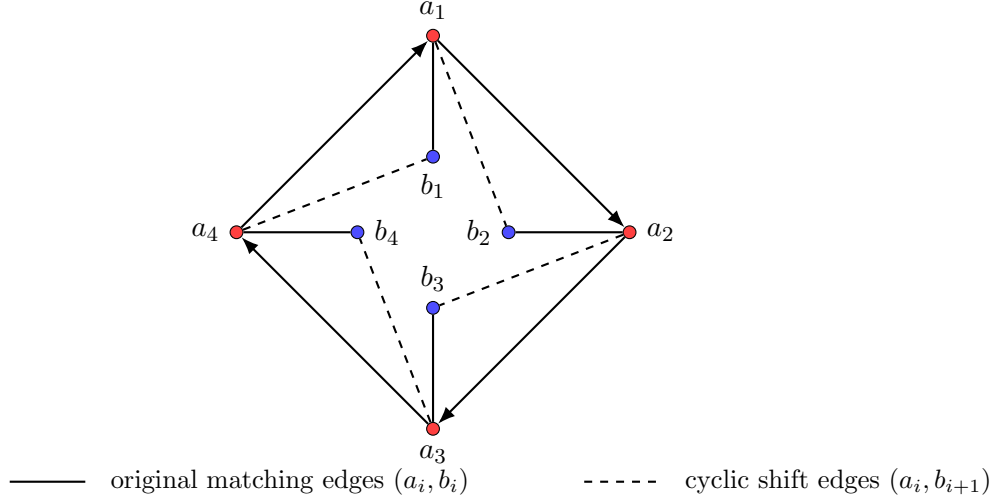

\begin{theorem}
\label{thm:algebraic-cycle-characterization}
 The matching $M = \{ (a_i,b_i) \}_{i=1}^n$ is a max-sum matching of $R \cup B$ if and only if $\operatorname{gain}(C) \le 0$ for every directed red cycle $C = (a_{i_1}, \dots, a_{i_m})$ (with $2 \le m \le n)$.
\end{theorem}

\begin{proof}
 Suppose that $M$ is max-sum. Let $C = (a_{i_1}, a_{i_2}, \dots, a_{i_m})$ be a directed red cycle. Consider the matching $M_C$ obtained from $M$ by fixing all pairs outside the cycle and replacing the pairs $(a_{i_1}, b_{i_1}), \dots, (a_{i_m}, b_{i_m})$ with $(a_{i_1},b_{i_2}), \dots, (a_{i_{m-1}},b_{i_m}), (a_{i_m}, b_{i_1})$. Define \[ \operatorname{cost}(M) = \sum_{(a,b) \in M} d(a,b). \] Since $M$ is max-sum, we obtain \[ \operatorname{cost}(M_C) \le \operatorname{cost}(M) ~ \Longleftrightarrow ~ \sum_{j=1}^m d(a_{i_j},b_{i_{j+1}}) \le \sum_{j=1}^m d(a_{i_j},b_{i_j}) ~ \Longleftrightarrow ~ \operatorname{gain}(C) \le 0, \] proving the necessity.

 Conversely, suppose that $\operatorname{gain}(C) \le 0$ for every directed red cycle $C$. Let $M^*$ be any other red-blue perfect matching. Then there is a permutation $\sigma$ of $\{1, \dots, n\}$ such that \[ M^* = \{ (a_i, b_{\sigma(i)}) : 1 \le i \le n \}. \] Decompose $\sigma$ into disjoint cycles. If $(i_1,i_2,\dots,i_m)$ is one nontrivial cycle of $\sigma$, then its contribution to \[ \sum_{i=1}^n [d(a_i,b_{\sigma(i)}) - d(a_i,b_i)] \] is exactly \[ \sum_{j=1}^m [d(a_{i_j},b_{i_{j+1}}) - d(a_{i_j},b_{i_j})] = \operatorname{gain}(a_{i_1}, a_{i_2}, \dots, a_{i_m}), \] where $i_{m+1} = i_1$. Fixed points of $\sigma$ contribute $0$. Therefore, by decomposing $\sigma$ into disjoint cycles and using our assumption, we have \[ \sum_{i=1}^n [d(a_i,b_{\sigma(i)}) - d(a_i,b_i)] = \sum_{i=1}^n d(a_i,b_{\sigma(i)}) - \sum_{i=1}^n d(a_i,b_i) \le 0 \] from which \[ \sum_{i=1}^n d(a_i,b_{\sigma(i)}) \le \sum_{i=1}^n d(a_i,b_i). \] Since $M^*$ was arbitrary, we conclude that $M$ is max-sum.
\end{proof} 

Note that Theorem \ref{thm:algebraic-cycle-characterization} does not need $d$ to be continuous, the properties P1-P4, or even to be a semimetric.

Geometrically, having a non-empty cyclic intersection of $H$-sets is a sufficient condition for a matching to be max-sum, as we now show.

\begin{lemma}
\label{lem:H-intersection-max-sum}
 Let $M = \{ (a_i,b_i) \}_{i=1}^n$ be a matching of $R \cup B$. If \[ \bigcap_{j=1}^m H(a_{i_j},a_{i_{j+1}}) \neq \emptyset \] for every directed red cycle $C = (a_{i_1}, \dots, a_{i_m})$ (with $2 \le m \le n)$, then $M$ is a max-sum matching. See Figure \ref{fig:cyclic-certificate}.
\end{lemma}

\begin{figure}[t]
\centering
\begin{tikzpicture}[scale=1.0]

\node[redpt, label=above:$a_1$] (a1) at (0,3.1) {};
\node[redpt, label=below left:$a_2$] (a2) at (-2.9,-1.6) {};
\node[redpt, label=below right:$a_3$] (a3) at (2.9,-1.6) {};

\draw[dircycle] (a1) -- (a3);
\draw[dircycle] (a3) -- (a2);
\draw[dircycle] (a2) -- (a1);

\fill[regionA] (0,0.8) ellipse (2.3 and 1.5);
\fill[regionB] (-1.1,-0.2) ellipse (2.2 and 1.4);
\fill[regionC] (1.1,-0.2) ellipse (2.2 and 1.4);

\draw[green!50!black] (0,0.8) ellipse (2.3 and 1.5);
\draw[orange!70!black] (-1.1,-0.2) ellipse (2.2 and 1.4);
\draw[cyan!60!black] (1.1,-0.2) ellipse (2.2 and 1.4);

\node[circle, fill=black, inner sep=1.5pt, label=below:$x$] at (0,0.2) {};

\node[lab] at (0,1.8) {$H(a_1,a_2)$};
\node[lab] at (-1.6,-1.2) {$H(a_2,a_3)$};
\node[lab] at (1.6,-1.2) {$H(a_3,a_1)$};

\node[note, align=center] at (0,-3.1)
{If \(x\in H(a_1,a_2)\cap H(a_2,a_3)\cap H(a_3,a_1)\),\\
the three inequalities telescope and yield \(\operatorname{gain}(C)\le 0\).};

\end{tikzpicture}
\caption{A cyclic intersection certificate for a directed triangle \(C=(a_1,a_2,a_3)\).}
\label{fig:cyclic-certificate}
\end{figure}

\begin{proof}
 By Theorem \ref{thm:algebraic-cycle-characterization}, it is enough to prove that $\operatorname{gain}(C) \le 0$ for every directed red cycle $C$. Fix $C = (a_{i_1}, a_{i_2}, \dots, a_{i_m})$. Choose \[ x \in \bigcap_{j=1}^m H(a_{i_j},a_{i_{j+1}}). \] Then for every $j = 1, \dots, m$, \[ x \in H(a_{i_j}, a_{i_{j+1}}) ~ \Longrightarrow ~ d(a_{i_j},b_{i_{j+1}}) - d(a_{i_{j+1}},b_{i_{j+1}}) \le d(a_{i_j},x) - d(a_{i_{j+1}},x). \] Adding these $m$ inequalities yields \[ \sum_{j=1}^m [d(a_{i_j},b_{i_{j+1}}) - d(a_{i_{j+1}},b_{i_{j+1}})] \le \sum_{j=1}^m [d(a_{i_j},x) - d(a_{i_{j+1}},x)] = 0. \] Noting that $i_{m+1} = i_1$ and relabeling, we obtain \[ \sum_{j=1}^m [d(a_{i_j},b_{i_{j+1}}) - d(a_{i_j},b_{i_j})] \le 0. \] That is, $\operatorname{gain}(C) \le 0$. Since $C$ was arbitrary, $\operatorname{gain}(C) \le 0$ for every directed red cycle $C$. Therefore, by Theorem \ref{thm:algebraic-cycle-characterization}, $M$ is max-sum.
\end{proof} 

\subsection{Bichromatic matching equilibrium}
\label{subsec:bic-2}

We say that a matching $M = \{ (a_i,b_i) : a_i \in R, b_i \in B, 1 \le i \le n \}$ is \textit{balanced} or \textit{in equilibrium} if the total sum of the lengths defined by the segments $a_i b_i$ is invariant under any permutation of points of the same color. Rigorously, given a weight function $d : R \times B \to \mathbb{R}$ in the plane, we say $M$ is balanced if \[ \sum_{i = 1}^n d(a_i,b_i) = \sum_{i = 1}^n d(a_i,b_{\sigma(i)}) \] for each of the $n!$ permutations $\sigma : \{ 1, 2, \dots, n \} \to \{ 1, 2, \dots, n \}$.

First, we characterize the property of equilibrium in a given matching using cycle-gains.

\begin{theorem}
\label{thm:balance-cycle-characterization}
 The matching $M = \{ (a_i,b_i) : 1 \le i \le n \}$ is balanced if and only if $\operatorname{gain}(C) = 0$ for every directed red cycle $C = (a_{i_1},a_{i_2},\dots,a_{i_m})$ with $2 \le m \le n$.
\end{theorem}

\begin{proof}
 Suppose first that $M$ is balanced. Let $C = (a_{i_1},a_{i_2},\dots,a_{i_m})$ be a directed red cycle. Let $M_C$ be the matching obtained from $M$ by keeping all pairs outside $C$ fixed and replacing \[ (a_{i_1},b_{i_1}), \dots, (a_{i_m},b_{i_m}) \quad \text{with} \quad (a_{i_1},b_{i_2}), \dots, (a_{i_{m-1}},b_{i_m}), (a_{i_m},b_{i_1}). \] Since $M$ is balanced, $M$ and $M_C$ have the same total weight. Cancellation of the pairs outside $C$ yields \[ \sum_{j=1}^m d(a_{i_j},b_{i_{j+1}}) = \sum_{j=1}^m d(a_{i_j},b_{i_{j}}). \] Therefore, $\operatorname{gain}(C) = 0$.

 Conversely, suppose that $\operatorname{gain}(C) = 0$ for every directed red cycle $C$. Let $M^*$ be any red-blue perfect matching. Then there exists a permutation $\sigma$ of $\{1, \dots, n\}$ such that $M^* = \{ (a_i,b_{\sigma(i)}) : 1 \le i \le n \}$. Decompose $\sigma$ into disjoint cycles. Each nontrivial cycle contributes exactly the gain of the corresponding directed red cycle, and fixed points contribute 0. Since every cycle gain is 0, we obtain \[ \sum_{i=1}^n [d(a_i,b_{\sigma(i)}) - d(a_i,b_i)] = 0. \] Hence \[ \sum_{i=1}^n d(a_i,b_{\sigma(i)}) = \sum_{i=1}^n d(a_i,b_i). \] Since $M^*$ was arbitrary, every red-blue perfect matching has the same total weight. Thus $M$ is balanced.
\end{proof} 

Now, we characterize the equilibrium property in terms of the distance matrix of the matching.

\begin{theorem}
\label{thm:balance-distance-matrix-characterization}
 Let $D = (D_{ij})$ be the $n \times n$ distance matrix defined by $D_{ij} = d(a_i,b_j)$. The following are equivalent:
 \begin{enumerate}
     \item $M$ is balanced.
     \item For every $i,k \in \{1, \dots, n\}$ and every $j,\ell \in \{1,\dots,n\}$, \[ D_{ij} + D_{k\ell} = D_{i\ell} + D_{kj}. \]
     \item There exist real numbers $\alpha_1, \dots, \alpha_n$ and $\beta_1, \dots, \beta_n$ such that $D_{ij} = \alpha_i + \beta_j$ for every $i,j$. Equivalently, \[ d(a_i,b_j) = \alpha_i + \beta_j \] for every red point $a_i$ and every blue point $b_j$.
 \end{enumerate}
\end{theorem}

\begin{proof}
 We prove (1) $\Rightarrow$ (2). Assume $M$ is balanced. Fix $i \neq k$ and $j \neq \ell$. Choose two permutations $\sigma$ and $\tau$ which agree everywhere except at $i$ and $k$, and such that \[ \sigma(i) = j, \quad \sigma(k) = \ell, \] while \[ \tau(i) = \ell, \quad \tau(k) = j. \] Since $M$ is balanced, the matching defined by $\sigma$ and the matching defined by $\tau$ have the same total weight. Moreover, all terms cancel except those involving $i$ and $k$. Therefore, \[ d(a_i,b_j) + d(a_k,b_\ell) = d(a_i,b_\ell) + d(a_k,b_j). \] Thus \[ D_{ij} + D_{k\ell} = D_{i\ell} + D_{kj}. \]

 Now we prove (2) $\Rightarrow$ (3). Assume every $2 \times 2$ alternating difference vanishes. Define  \[ \alpha_i = D_{i1}, \quad \beta_j = D_{1j} - D_{11}. \] Using condition (2) with indices $i,1,j,1$, we obtain \[ D_{ij} + D_{11} = D_{i1} + D_{1j}. \] Hence \[ D_{ij} = D_{i1} + D_{1j} - D_{11} = \alpha_i + \beta_j. \] So $D$ has additive form.

 Finally, suppose (3) holds. For any permutation $\sigma$, \[ \sum_{i=1}^n D_{i,\sigma(i)} = \sum_{i=1}^n (\alpha_i + \beta_{\sigma(i)}) = \sum_{i=1}^n \alpha_i + \sum_{j=1}^n \beta_j. \] Note that the right-hand side is independent of $\sigma$. Therefore every red-blue perfect matching has the same total weight. Hence $M$ is balanced.
\end{proof} 

Finally, we can characterize the equilibrium property in terms of certain level-sets, thus adding a geometric interpretation of the balance. Assume that the weight function is induced by a function $d : \mathbb{R}^2 \times \mathbb{R}^2 \to \mathbb{R}$. For $i,k \in \{ 1, \dots, n \}$, let us define the red-pair distance-difference function \[ \Phi_{ik}(x) = d(a_i,x) - d(a_k,x). \] For $j,\ell \in \{ 1, \dots, n \}$, let us define the blue-pair distance-difference function \[ \Psi_{j\ell}(x) = d(x,b_j) - d(x,b_\ell). \]

\begin{theorem}
\label{thm:balance-level-set-characterization}
 The matching $M$ is balanced if and only if, for every pair of red points $a_i,a_k$, all blue points lie on a common level set of $\Phi_{ik}$. That is, for every $i,k$, there exists a real number $\lambda_{ik}$ such that \[ B \subset \{ x \in \mathbb{R}^2 : \Phi_{ik}(x) = \lambda_{ik} \}. \] Equivalently, $M$ is balanced if and only if, for every pair of blue points $b_j,b_\ell$, all red points lie on a common level set of $\Psi_{j\ell}$. That is, for every $j,\ell$, there exists a real number $\mu_{j\ell}$ such that \[ R \subset \{ x \in \mathbb{R}^2 : \Psi_{j\ell}(x) = \mu_{j\ell} \}. \]
\end{theorem}

\begin{figure}[t]
\centering
\begin{tikzpicture}[scale=1.0]

\begin{scope}[xshift=-3.5cm]
  \node[note, text width=5.5cm, align = center] at (0,-3.5) {(a) Red pair determines a common level set for blue points};

  \node[redpt, label=below:$a_i$] (ai) at (-2.2,0) {};
  \node[redpt, label=below:$a_k$] (ak) at ( 2.2,0) {};

  \draw[levcurve, blue!70!black]
    plot[domain=-1.65:1.65, samples=120, smooth]
      ({1.2 + 0.35*\x*\x}, {\x});

  \node[bluept, label=right:$b_1$] at (1.50, -1.10) {};
  \node[bluept, label=right:$b_2$] at (1.20,  0.00) {};
  \node[bluept, label=right:$b_3$] at (1.50,  1.10) {};

  \node[lab, align=center] at (0,-2.4)
    {\(\Phi_{ik}(x)=d(a_i,x)-d(a_k,x)=\lambda_{ik}\)};
\end{scope}

\begin{scope}[xshift=3.5cm]
  \node[note, text width=5.5cm, align=center] at (0,-3.5) {(b) Blue pair determines a common level set for red points};

  \node[bluept, label=below:$b_j$] (bj) at (-2.2,0) {};
  \node[bluept, label=below:$b_\ell$] (bl) at ( 2.2,0) {};

  \draw[levcurve, red!70!black]
    plot[domain=-1.65:1.65, samples=120, smooth]
      ({-1.2 - 0.35*\x*\x}, {\x});

  \node[redpt, label=left:$a_1$] at (-1.50, -1.10) {};
  \node[redpt, label=left:$a_2$] at (-1.20,  0.00) {};
  \node[redpt, label=left:$a_3$] at (-1.50,  1.10) {};

  \node[lab, align=center] at (0,-2.4)
    {\(\Psi_{j\ell}(x)=d(x,b_j)-d(x,b_\ell)=\mu_{j\ell}\)};
\end{scope}

\end{tikzpicture}
\caption{Equilibrium as a level-set condition.}
\label{fig:equilibrium-level-sets}
\end{figure}
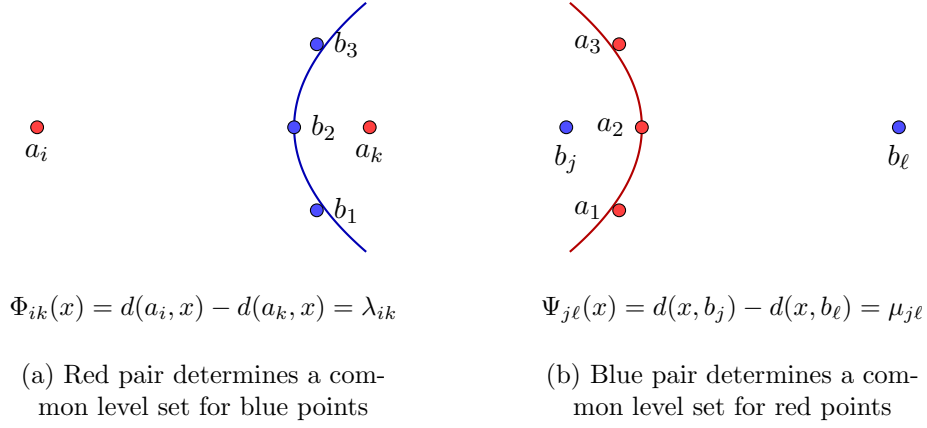

\begin{proof}
 By Theorem \ref{thm:balance-distance-matrix-characterization}, $M$ is balanced if and only if \[ d(a_i,b_j) + d(a_k,b_\ell) = d(a_i,b_\ell) + d(a_k,b_j) \] for every $i,j,k,\ell$. Rearranging gives \[ d(a_i,b_j) - d(a_k,b_j) = d(a_i,b_\ell) - d(a_k,b_\ell). \] In terms of $\Phi_{ik}$, this says \[ \Phi_{ik}(b_j) = \Phi_{ik}(b_\ell) \] for every pair of blue points $b_j,b_\ell$. Therefore, for each fixed red pair $a_i,a_k$, the function $\Phi_{ik}$ is constant on $B$. Equivalently, all blue points lie on a common level set of $\Phi_{ik}$. See Figure \ref{fig:equilibrium-level-sets}(a).

 Similarly, the same equality can be arranged as \[ d(a_i,b_j) - d(a_i,b_\ell) = d(a_k,b_j) - d(a_k,b_\ell). \] In terms of $\Psi_{j\ell}$, this says \[ \Psi_{j\ell}(a_i) = \Psi_{j\ell}(a_k) \] for every pair of red points $a_i,a_k$. Therefore, for each fixed blue pair $b_j,b_\ell$, the function $\Psi_{j\ell}$ is constant on $R$. Equivalently, all red points lie on a common level set of $\Psi_{j\ell}$. See Figure \ref{fig:equilibrium-level-sets}(b).

 Thus the two level-set formulations are equivalent to $M$ being balanced.
\end{proof} 

In the language of $H$-sets, note that \[ H(a_i,a_k) = \{ x : \Phi_{ik}(x) \ge \Phi_{ik}(b_k) \} \] while \[ h(a_i,a_k) = \{ x : \Phi_{ik}(x) \le \Phi_{ik}(b_i) \}. \] Balancedness implies \[ \Phi_{ik}(b_1) = \Phi_{ik}(b_2) = \cdots = \Phi_{ik}(b_n). \] Therefore, in equilibrium, \[ \Phi_{ik}(b_i) = \Phi_{ik}(b_k), \] so the two boundary level sets \[ \Phi_{ik}(x) = \Phi_{ik}(b_i) \quad \text{and} \quad \Phi_{ik}(x) = \Phi_{ik}(b_k) \] coincide. More strongly, every blue point lies on the same level set: \[ B \subset \{ x : \Phi_{ik}(x) = \Phi_{ik}(b_i) \} \] Thus, for every red pair $a_i,a_k$, all blue points lie on the common boundary level set between $H(a_i,a_k)$ and $h(a_i,a_k)$. 

\subsubsection{Squared Euclidean distance}
\label{subsubsec:bic-2-1}

Assume now that $d(p,q) = \|p - q\|^2$. Let \[ U_R = \operatorname{span}\{ a_i - a_k : 1 \le i,k \le n \} \quad \text{and} \quad U_B = \operatorname{span}\{ b_j - b_\ell : 1 \le j,\ell \le n \}. \]

\begin{corollary}
\label{cor:balance-euclidean-squared}
 For the squared Euclidean distance, $M$ is balanced if and only if $U_R \perp U_B$. In other words, \[ (a_i - a_k) \cdot (b_j - b_\ell) = 0 \] for every $i,j,k,\ell$.
\end{corollary}

\begin{figure}[t]
\centering
\begin{tikzpicture}[scale=1.0]

\begin{scope}[xshift=0cm, yshift=0cm]
  \node[note] at (0,-3.0) {(a) Noncollapsed balanced configuration};

  \draw[gray] (-2.5,0) -- (2.5,0);
  \node[redpt, label=below:$a_1$] at (-1.8,0) {};
  \node[redpt, label=below:$a_2$] at (0,0) {};
  \node[redpt, label=below:$a_3$] at (1.8,0) {};

  \draw[gray] (0,-2.3) -- (0,2.3);
  \node[bluept, label=left:$b_1$] at (0,-1.6) {};
  \node[bluept, label=left:$b_2$] at (0, 0.3) {};
  \node[bluept, label=left:$b_3$] at (0, 1.8) {};

  \draw[axisarrow, red!70!black] (-1.2,0.45) -- (1.2,0.45);
  \node[lab, red!70!black] at (-0.5,0.8) {$U_R$};

  \draw[axisarrow, blue!70!black] (0.45,-1.2) -- (0.45,1.2);
  \node[lab, blue!70!black] at (1.0,-0.5) {$U_B$};
\end{scope}

\begin{scope}[xshift=-3.5cm, yshift=-6.4cm]
  \node[note] at (0,-3.0) {(b) Noncollinear \(R\) forces blue collapse};

  \node[redpt, label=above:$a_1$] at (0,2.0) {};
  \node[redpt, label=below left:$a_2$] at (-1.8,-1.2) {};
  \node[redpt, label=below right:$a_3$] at (1.8,-1.2) {};

  \node[bluept, label=right:{$b_1=b_2=b_3$}] at (0,0.2) {};
\end{scope}

\begin{scope}[xshift=3.5cm, yshift=-6.4cm]
  \node[note] at (0,-3.0) {(c) Noncollinear \(B\) forces red collapse};

  \node[bluept, label=above:$b_1$] at (0,2.0) {};
  \node[bluept, label=below left:$b_2$] at (-1.8,-1.2) {};
  \node[bluept, label=below right:$b_3$] at (1.8,-1.2) {};

  \node[redpt, label=right:{$a_1=a_2=a_3$}] at (0,0.2) {};
\end{scope}

\end{tikzpicture}
\caption{Squared Euclidean equilibrium.}
\label{fig:squared-euclidean}
\end{figure}

\begin{proof}
 For squared Euclidean distance, \[ D_{ij} = \| a_i - b_j \|^2. \] The $2 \times 2$ alternating difference is \[ D_{ij} + D_{k\ell} - D_{i\ell} - D_{kj}. \] Expanding, \[ \|a_i - b_j\|^2 + \|a_k - b_\ell\|^2 - \|a_i - b_\ell\|^2 - \|a_k - b_j\|^2 = -2(a_i - a_k) \cdot (b_j - b_\ell). \] By Theorem \ref{thm:balance-distance-matrix-characterization}, $M$ is balanced if and only if all these alternating differences vanish. Therefore $M$ is balanced if and only if \[ (a_i - a_k) \cdot (b_j - b_\ell) = 0 \] for every $i,j,k,\ell$, which is equivalent to $U_R \perp U_B$.
\end{proof} 

Consequently, by Corollary \ref{cor:balance-euclidean-squared}, in the squared Euclidean case:

\begin{enumerate}
    \item If $R$ contains three noncollinear points, then all blue points coincide. See Figure \ref{fig:squared-euclidean}(b).

    \item If $B$ contains three noncollinear points, then all red points coincide. See Figure \ref{fig:squared-euclidean}(c).

    \item If both color classes contain at least two distinct points and neither color class collapses to one point, then both $R$ and $B$ are collinear, and their supporting directions are perpendicular. See Figure \ref{fig:squared-euclidean}(a). Conversely, every configuration satisfying this orthogonality condition is balanced.
\end{enumerate}

\subsubsection{Euclidean distance}
\label{subsubsec:bic-2-2}

Assume now that $d(p,q) = \|p - q\|$. The main issue with this metric is that it yields hyperbolic level sets rather than affine hyperplanes. To characterize the equilibrium of red-blue matchings $M$, we focus on the number of \emph{distinct} red and blue locations, rather than the cardinality of $R$ and $B$. 

Clearly, if $R$ has exactly one distinct location or $B$ has exactly one distinct location, then $M$ is automatically balanced. Indeed, if all red points coincide at, say, $a$, then \[ \sum_{i} \| a_i - b_{\sigma(i)} \| = \sum_i \| a - b_{\sigma(i)} \| = \sum_j \| a - b_j \|, \] which is independent of $\sigma$. The blue-collapse case is symmetric.

Suppose $R$ collapses into two distinct points $p$ and $q$. By Theorem \ref{thm:balance-level-set-characterization}, $M$ is balanced if and only if all blue points lie on one common level set \[ \{ x \in \mathbb{R}^2 : \|p - x\| - \|q - x\| = \lambda \} \] for some real number $\lambda$. Note that, geometrically, this level set is
\begin{itemize}
    \item a branch of a hyperbola with foci $p,q$, if $0 < |\lambda| < \|p-q\|$;
    \item the perpendicular bisector of segment $pq$, if $\lambda = 0$;
    \item one of the two exterior rays of the line $pq$, if $|\lambda| = \|p-q\|$.
\end{itemize} The symmetric statement holds if $B$ has exactly two distinct locations: if $B$ collapses into two distinct points $p'$ and $q'$, then $M$ is balanced if and only if all red points lie on one common level set \[ \{ x \in \mathbb{R}^2 : \|x - p'\| - \|x - q'\| = \mu \}. \]

Before we study the characterization for the general case, we state and prove two elementary facts about Euclidean distance-difference level sets.

\begin{lemma}
\label{lem:tech-1}
 If $p,q,r \in \mathbb{R}^2$ are noncollinear, then for any real numbers $\lambda,\mu$, the system \[ \begin{cases} \|p-x\| - \|q-x\| = \lambda, \\ \|p-x\| - \|r-x\| = \mu \end{cases} \] has at most two solutions $x \in \mathbb{R}^2$.
\end{lemma}

\begin{proof}
Put $\rho = \|p-x\|$. Then \[
\|q-x\| = \rho - \lambda, \qquad \|r-x\| = \rho - \mu. \]
Squaring and subtracting $\|p-x\|^2 = \rho^2$ from each equation gives \[ 2(q-p) \cdot x = \|q\|^2 - \|p\|^2 + 2\lambda\rho - \lambda^2, \] and \[ 2(r-p) \cdot x
= \|r\|^2 - \|p\|^2 + 2\mu\rho - \mu^2.\] Since $p,q,r$ are noncollinear, the vectors $q-p$ and $r-p$ are linearly independent. Hence these two linear equations determine $x$ as an affine function of $\rho$. Substituting this expression into $\|p-x\|^2 = \rho^2$ gives a quadratic equation in $\rho$. Therefore there are at most two possible values of $\rho$, and hence at most two possible points $x$.
\end{proof}

\begin{lemma}
\label{lem:tech-2}
 Let $u \neq v$, and let $L$ be a line. If the level set \[ \{ x : \|x - u\| - \|x - v\| = \lambda \} \] contains three distinct points in $L$, then either:
 \begin{enumerate}
     \item $\lambda = 0$, and $L$ is the perpendicular bisector of the segment $uv$; or
     \item $|\lambda| = \|u-v\|$, $L$ is the line through $u$ and $v$, and the three points lie on one of the two exterior rays determined by $u$ and $v$.
 \end{enumerate}
\end{lemma}

\begin{proof}
The level set $\{x : \|x-u\| - \|x-v\| = \lambda\}$ is empty if $|\lambda| > \|u-v\|$. If $\lambda = 0$, it is the perpendicular bisector of $uv$. If
$0 < |\lambda| < \|u-v\|$, it is a branch of a nondegenerate hyperbola with foci $u,v$. Hence any line meets it in at most two points. Finally, if $|\lambda| = \|u-v\|$, equality in the reverse triangle inequality forces $x,u,v$ to be collinear, with $x$ lying on one of the two exterior rays determined
by $u$ and $v$.

Therefore, if such a level set contains three distinct points of a line $L$, the level set must be one of the two degenerate cases: either $\lambda = 0$ and $L$ is the perpendicular bisector of $uv$, or $|\lambda| = \|u-v\|$, $L$ is the line through $u$ and $v$, and the three points lie on one exterior ray.
\end{proof}

\begin{theorem}
\label{thm:balance-euclidean-3}
 Suppose $R$ and $B$ contain at least three distinct points each. $M$ is balanced if and only if all red points and blue points of $R \cup B$ lie on a common line $L$, and the two color classes are linearly separated on $L$. That is, after choosing an affine coordinate $t : L \to \mathbb{R}$, either \[ t(a_i) \le t(b_j), \quad \text{for all}~ i,j, \] or \[ t(b_j) \le t(a_i), \quad \text{for all}~ i,j. \] Equivalently, the convex hulls of the two color classes are intervals on the same line with disjoint interiors, possibly sharing one endpoint.
\end{theorem}

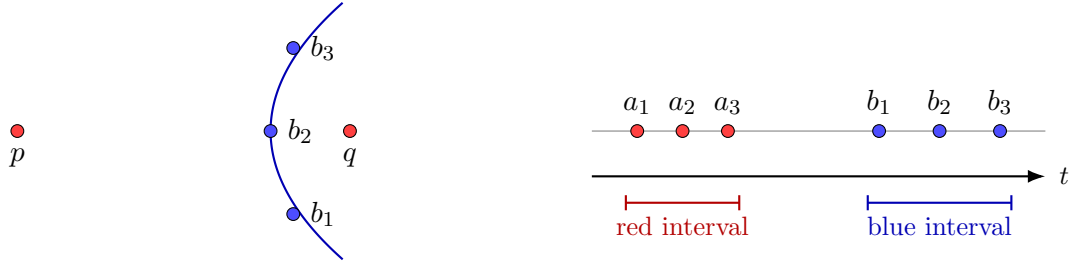
\begin{figure}[t]
\centering
\begin{tikzpicture}[scale=1.0]

\begin{scope}[xshift=-4.2cm]
  \node[note] at (0,-3.2) {(a) Two red locations: blue points on one level set};

  \node[redpt, label=below:$p$] at (-2.2,0) {};
  \node[redpt, label=below:$q$] at ( 2.2,0) {};

  \draw[levcurve, blue!70!black]
    plot[domain=-1.7:1.7, samples=120, smooth]
      ({1.15 + 0.33*\x*\x}, {\x});

  \node[bluept, label=right:$b_1$] at (1.45,-1.10) {};
  \node[bluept, label=right:$b_2$] at (1.15,0.00) {};
  \node[bluept, label=right:$b_3$] at (1.45,1.10) {};

  \node[lab, align=center] at (0,-2.4)
    {\(\|p-x\|-\|q-x\|=\lambda\)};
\end{scope}

\begin{scope}[xshift=4.2cm]
  \node[note] at (0,-3.2) {(b) At least three distinct points in each color class};

  \draw[gray] (-3.0,0) -- (3.0,0);
  \draw[axisarrow] (-3.0,-0.6) -- (3.0,-0.6);
  \node[lab] at (3.25,-0.6) {$t$};

  \node[redpt, label=above:$a_1$] at (-2.4,0) {};
  \node[redpt, label=above:$a_2$] at (-1.8,0) {};
  \node[redpt, label=above:$a_3$] at (-1.2,0) {};

  \node[bluept, label=above:$b_1$] at (0.8,0) {};
  \node[bluept, label=above:$b_2$] at (1.6,0) {};
  \node[bluept, label=above:$b_3$] at (2.4,0) {};

  \draw[red!70!black, thick] (-2.55,-0.95) -- (-1.05,-0.95);
  \draw[red!70!black, thick] (-2.55,-1.05) -- (-2.55,-0.85);
  \draw[red!70!black, thick] (-1.05,-1.05) -- (-1.05,-0.85);
  \node[lab, red!70!black] at (-1.8,-1.25) {red interval};

  \draw[blue!70!black, thick] (0.65,-0.95) -- (2.55,-0.95);
  \draw[blue!70!black, thick] (0.65,-1.05) -- (0.65,-0.85);
  \draw[blue!70!black, thick] (2.55,-1.05) -- (2.55,-0.85);
  \node[lab, blue!70!black] at (1.6,-1.25) {blue interval};

  \node[lab, align=center] at (0,-2.4)
    {\(t(a_i)\le t(b_j)\) for all \(i,j\)};
\end{scope}

\end{tikzpicture}
\caption{Euclidean equilibrium.}
\label{fig:euclidean-equilibrium}
\end{figure}

\begin{proof}
 We first prove necessity. Assume that $M$ is balanced. By Theorem \ref{thm:balance-level-set-characterization}, for every pair of red points $p,q \in R$, all blue points lie on a common level set of \[ x \mapsto \|p-x\| - \|q-x\|. \] See Figure \ref{fig:euclidean-equilibrium}(a). Equivalently, for every pair of blue points $u,v \in B$, all red points lie on a common level set of \[ x \mapsto \|x-u\| - \|x-v\|. \] Suppose, for a contradiction, that the red points are not collinear. Since $R$ contains at least three distinct points, we may choose three noncollinear red points $p,q,r \in R$. By balancedness, all blue points lie on one level set of \[ x \mapsto \|p-x\| - \|q-x\|, \] and also on one level set of \[ x \mapsto \|p-x\| - \|r-x\|. \] By Lemma \ref{lem:tech-1}, these two levels sets meet at two points at most. Therefore $B$ contains at most two distinct points, contradicting the hypothesis that $B$ contains at least three distinct points. Hence $R$ must be collinear. By symmetry, $B$ must also be collinear.

 Let $L_R$ be the line containing $R$ and let $L_B$ be the line containing $B$. We prove that $L_R = L_B$.

 Choose two distinct blue points $u,v \in B$. Since $M$ is balanced, all red points lie on one level set of \[ x \mapsto \|x-u\| - \|x-v\|. \] Since $R$ contains at least three distinct points on the line $L_R$, Lemma \ref{lem:tech-2} implies that 
 \begin{enumerate}
     \item $L_R$ is the perpendicular bisector of $uv$; or
     \item $L_R$ is the line through $u$ and $v$.
 \end{enumerate} Clearly, if the second alternative occurs for some pair $u,v$, then $L_R = L_B$. Suppose instead that the first alternative occurs for every pair of distinct points $u,v \in B$. Since $B$ contains at least three distinct collinear points, the perpendicular bisectors of its different pairs cannot all be the same line, as $L_R$ is fixed. Therefore, this must occur for some pair of blue points, and hence $L_R = L_B$. Thus all red and blue points lie on a common line $L$.

 To prove that the two color classes are linearly separated, choose an affine coordinate $t$ on $L$. Let $u,v \in B$ with minimal and maximal $t$-coordinates among the distinct blue points, so $t(u) < t(v)$. By balancedness, all red points lie on one level set of \[ x \mapsto \|x-u\| - \|x-v\|. \] Restricted to the line $L$, we get \[ \|x-u\| - \|x-v\| = \begin{cases} t(u) - t(v), & t(x) \le t(u), \\ 2t(x) - t(u) - t(v), & t(u) \le t(x) \le t(v), \\ t(v) - t(u), & t(x) \ge t(v). \end{cases} \] On the interval $[t(u),t(v)]$, this function is strictly increasing, so each level set meets that interval in at most one point. On the left exterior ray $t(x) \le t(u)$, the function is constantly $t(u)-t(v)$, while on the right exterior ray $t(x) \ge t(v)$, it is constantly $t(v)-t(u)$. Since $t(u) < t(v)$, these two constants are distinct. Hence a single level set cannot contain points on both exterior rays. Because $R$ contains at least three distinct points and all red points lie on the same level set, all red points must lie on one closed exterior ray: \[ t(a_i)\le t(u) \quad \text{for all } i, \qquad \text{or} \qquad t(a_i)\ge t(v)\quad\text{for all } i.\] Since $u$ and $v$ are the extreme blue points, this means either \[ t(a_i) \le t(b_j) \quad \text{for all}~ i,j, \qquad \text{or} \qquad t(b_j) \le t(a_i) \quad \text{for all}~ i,j. \] Thus the two color classes are linearly separated. See Figure \ref{fig:euclidean-equilibrium}(b). This proves necessity.

 To prove sufficiency, suppose that all points of $R \cup B$ lie on a line $L$, and the two color classes are linearly separated. Choose an affine coordinate $t$ on $L$. Assume, without loss of generality, that $t(a_i) \le t(b_j)$ for all $i,j$. Then \[ \|a_i - b_j\| = t(b_j) - t(a_i). \] By defining $\alpha_i = -t(a_i)$, $\beta_j = t(b_j)$, we get \[ \|a_i - b_j\| = \alpha_i + \beta_j \] for every $i,j$. Hence the Euclidean distance matrix has additive form. By Theorem \ref{thm:balance-distance-matrix-characterization}, $M$ is balanced. The case $t(b_j) \le t(a_i)$ for all $i,j$ is identical, using \[ \|a_i - b_j\| = t(a_i) - t(b_j). \] Therefore $M$ is balanced.
\end{proof}

{\small


\bibliographystyle{abbrv}
\bibliography{main}

}

\end{document}